\documentclass[12pt]{amsart}
\usepackage{amsmath,amsfonts,amscd,bezier,amsthm,amssymb,enumerate,enumitem}
\usepackage{mathrsfs,mathtools}
\usepackage{color}
\usepackage{cite, hyperref}
\usepackage{graphicx}
\usepackage{float}
\usepackage{enumitem,xfrac,nccmath}
\usepackage{tasks, multicol,yhmath}

\usepackage[pagewise]{lineno}

\newtheorem{theorem}{Theorem}[section]
\newtheorem{lemma}[theorem]{Lemma}

\theoremstyle{definition}

\newtheorem{remark}[theorem]{Remark}
\newtheorem{example}[theorem]{Example}
\newtheorem{proposition}[theorem]{Proposition}

\newtheorem{problem}{Problem}

\DeclareMathOperator{\dist}{dist}
\DeclareMathOperator{\supp}{supp}
\DeclareMathOperator{\lcm}{lcm}

\newcommand{\wh}{\widehat}

\newcommand{\wt}{\widetilde}

\newcommand{\calD}{\mathcal{D}}

\newcommand{\calN}{\mathcal{N}}

\newcommand{\R}{\mathbb{{R}}}
\newcommand{\Z}{\mathbb{{Z}}}
\newcommand{\N}{\mathbb{{N}}}
\newcommand{\T}{\mathbb{{T}}}
\newcommand{\C}{\mathbb{{C}}}
\newcommand{\D}{\mathbb{{D}}}

\DeclareMathOperator{\vspan}{span}
\newcommand{\Abs}[1]{\left\lvert#1\right\rvert}
\newcommand{\wb}{\overline}
\newcommand{\ls}{\lesssim}
\newcommand{\gs}{\gtrsim}

\numberwithin{equation}{section}

\allowdisplaybreaks

\makeatletter
\DeclareRobustCommand{\eqrefp}[2]{%
	\textup{\tagform@{\refp{#1}{#2}}}}
\makeatother
\DeclareRobustCommand{\refp}[2]{%
	\expandafter\ifx\csname r@#1\endcsname\relax
	\textbf{??}%
	\else
	\edef\areferencia{\ref{#1}-)}%
	\expandafter\eqrefpaux\areferencia-#2%
	\fi}
\def\eqrefpaux#1-#2){#1}

\begin{document}
	\title[Uniqueness with coefficient restrictions]{Uniqueness sets for functions of Dirichlet-type with restricted Taylor coefficients}

        \author[N. Miheisi]{Nazar Miheisi}
	\address[N. Miheisi]{Department of Mathematics, Faculty of Natural, Mathematical and Engineering Sciences, S4.02.1, Strand Building, Strand Campus, Strand, London WC2R 2LS.}
	\email{nazar.miheisi@kcl.ac.uk}

	\author[D. Seco]{Daniel Seco}
	\address[D. Seco]{Departamento de An\'alisis Matem\'atico e IMAULL,  Universidad de La Laguna, Avenida Astrof\'isico Francisco S\'anchez, s/n, 38206 San Crist\'obal de La Laguna, Santa Cruz de Tenerife,  Spain} \email{dsecofor@ull.edu.es}

\thanks{The second author was supported by Grant PID2024-160185NB-I00 funded by MICIU/AEI/10.13039/501100011033 and by FEDER, UE, through the Generaci\'on de Conocimiento programme of Spanish Ministry of Science, Innovation and Universities;  and by grant RYC2021-034744-I of the Ram\'on y Cajal programme from Agencia Estatal de Investigaci\'on (Spanish Ministry of Science, Innovation and Universities).}
	
\subjclass[2020]{Primary: 46E22; Secondary: 30H10, 30C15, 42C25.}
	
	\keywords{Reproducing kernel Hilbert spaces; spaces of analytic functions; uniqueness sets.}

	\date{\today}
	
	\begin{abstract} 
	Let $H$ be a reproducing kernel Hilbert space over the unit disk $\D$,
	where analytic monomials span a dense subset. Given $\mathcal{N}
	\subseteq\mathbb Z_+$ and $\Lambda\subseteq\mathbb \D$ we say that
	$(\Lambda,\calN)$ is a uniqueness pair for $H$ if $\Lambda$
	is a uniqueness set for the subspace of $H$ spanned by
	$\{z^n:\;n\in\calN\}$. We examine uniqueness pairs in the Dirichlet-type spaces $\calD_\alpha$, $0\leq\alpha\leq1$. We prove two complementary
	results. First, if $\mathcal N$ contains sufficiently long finite
	arithmetic progressions with fixed gap size, then no sequence $\Lambda$
	tending sufficiently rapidly to the boundary forms a uniqueness pair
	with $\calN$. Second, if $\calN$ satisfies a suitable arithmetic
	sparsity condition then one can construct uniqueness pairs
	$(\Lambda,\calN)$ with the points of $\Lambda$ tending to the boundary
	arbitrarily fast.
	\end{abstract}
	
	\maketitle
	
\section{Introduction}

\subsection{Overview}

Consider a reproducing kernel Hilbert space $H$ of analytic functions defined
on the unit disk $\D=\{z\in\C:\;|z|<1\}$. We assume that the polynomials are
dense in $H$. A subset $\Lambda$ of $\D$ is a \emph{uniqueness set} for $H$
if the only function in $H$ vanishing on $\Lambda$ is the zero function. 
The study of uniqueness sets has a  long history. For the usual Hardy space
$H^2$, the uniqueness sets are characterized by the Blaschke condition:
$\Lambda$ is a uniqueness  set for $H^2$ if and only if
$$
\sum_{\lambda\in\Lambda} (1-|\lambda|) = \infty.
$$
For most other spaces, however, a complete characterization seems out of reach.
Nevertheless, for a large class of spaces $H$ -- in particular, many Dirichlet-type spaces -- it is still true that sequences approaching the
boundary sufficiently fast cannot be uniqueness sets.

A natural question is how the structure of uniqueness sets changes if one
restricts to functions generated by a prescribed set of monomials. More
precisely, given a set $\calN$ of non-negative integers, one may ask which
subsets $\Lambda\subseteq\D$ are uniqueness sets for $\wb{\vspan}\{ z^n:
\, n\in\calN\}$, and how this depends on $\calN$. This leads to the following
notion: we say that $(\Lambda,\calN)$ is a \emph{uniqueness pair} if the
only function in $\wb{\vspan}\{ z^n:\, n\in\calN\}$ vanishing on $\Lambda$
is the zero function.

Fixing $\calN$, a number of natural questions arise: do there exist non-trivial
uniqueness pairs? Can one find uniqueness pairs $(\Lambda,\calN)$ for which the
points of $\Lambda$ tend to the boundary arbitrarily fast? Are the uniqueness
pairs $(\Lambda,\calN)$ stable under small perturbations of the points in
$\Lambda$? These depend in a delicate way on the arithmetic structure and
density of the set $\calN$. In this note, we initiate the study of
these questions in Dirichlet-type spaces.

We remark that the concept of uniqueness pairs introduced here is closely
related to the notions of \emph{mutually annihilating pairs}
\cite[Chap. 3]{Hav-Jor} and \emph{Heisenberg uniqueness pairs}
\cite{HUP1, HUP2, HUP3} in harmonic analysis, but is more naturally
adapted to the setting of reproducing kernel Hilbert spaces. There are some recent articles regarding uniqueness and restrictions on coefficients that deserve to be mentioned \cite{Boscardin, Zannier}, even if they do not focus on Hilbert spaces.

\subsection{The setting: Dirichlet-type spaces}

For $0\leq\alpha\leq1$, the Dirichlet-type space $\calD_\alpha$ consists of all
analytic functions $f$ on $\D$ such that
$$
\|f\|_\alpha^2 = \sum_{k\geq0} (k+1)^\alpha|\wh{f}(k)|^2 <\infty.
$$
Here, and throughout, $\wh{f}(k)$ is the $k^\text{th}$ Taylor coefficient
of $f$. The case $\alpha=0$ is the Hardy space $H^2$, and $\alpha=1$ gives
the classical Dirichlet space $\calD$.
More generally, one may define $\calD_\alpha$ for all $\alpha\in \R$. However, when $\alpha>1$, every function in $\calD_\alpha$ extends continuously to
$\overline{\D}$. Even though a characterization of uniqueness pairs would be
very interesting, the questions considered in this paper become essentially
trivial. For this reason, we restrict attention to the range $0\leq\alpha\leq1$.
The classical references \cite{Gar, EFKMR} contain the essentials on
$\calD_\alpha$ spaces.

Each $\calD_\alpha$ is a reproducing kernel Hilbert space, and the
reproducing kernel at $\lambda\in\D$ is
\begin{equation}\label{eqn47}
k^\alpha_\lambda(z) = \sum_{k\geq0}\frac{\overline{\lambda}^k z^k}{(k+1)^\alpha}.
\end{equation}
The zero sets for the spaces $\calD_\alpha$ have been extensively studied,
and enjoy several useful structural properties. In particular, they coincide
with those of the multiplier algebra, from which it follows that a finite
union of zero sets is again a zero set.

\subsection{Main results}

Let $\calN$ be a set of non-negative integers and $\Lambda$ be a subset of $\D$.
We recall that $(\Lambda, \calN)$ is a \emph{uniqueness pair} (for $\calD_\alpha$)
if whenever $f\in \calD_\alpha$ satisfies $f \perp k_\lambda^\alpha$ for all
$\lambda\in\Lambda$ and $f \perp z^k$ for all $k \in \Z_+\setminus\calN$, we
must have $f \equiv 0$.

For now, we ignore the case when $\calN$ is finite. The simplest case, otherwise, is when $\calN$ contains an infinite arithmetic progression.
In this case, $(\Lambda, \calN)$ is a uniqueness pair if and only if $\Lambda$
is a uniqueness set for $\calD_\alpha$ (see Proposition \ref{prop:AP}). A
celebrated result of Shapiro and Shields \cite{ShaShi} asserts that in this case
$\Lambda$ must satisfy
\begin{equation}\label{eqnA01}
\sum_{\lambda\in\Lambda} \frac{1}{\|k^\alpha_\lambda\|_\alpha^2} = \infty.\end{equation}
Consequently, if the points of $\Lambda$ tend to the boundary sufficiently
quickly, $(\Lambda, \calN)$ is not a uniqueness pair. Our first main result
shows that this remains true when $\calN$ contains a sequence of finite
arithmetic progressions, with a fixed gap size, and suitably growing lengths.

\begin{theorem}\label{thm:long-AP}
Let $0\leq\alpha\leq1$, and let $\{a_k\}_{k\geq1}$ and $\{l_k\}_{k\geq1}$
be sequences of positive integers such that
\begin{equation}\label{eq:progression}
l_k/a_k^\alpha\to\infty \quad\text{as}\quad k\to\infty.
\end{equation}
Let $\calN$ be a subset of $\Z_+$ which contains each arithmetic progression
$$
\{a_k,\, a_k+q, \,\dots,\, a_k+(l_k-1)q\}, \quad k\geq1,
$$
for some fixed $q\in\N$.
Then there exists an increasing sequence
$\{t_k\}_{k\geq1}\subseteq(0,1)$ such that the following holds: if
$\Lambda=\{\lambda_k\}_{k\geq1}\subseteq\D$ satisfies $|\lambda_k|\geq t_k$
for all sufficiently large $k$, then $(\Lambda,\calN)$ is not a uniqueness
pair for $\calD_\alpha$.
\end{theorem}

When $\alpha=0$, the assumptions in Theorem \ref{thm:long-AP} simply reduce
to the requirement that $\calN$ contains arbitrarily long arithmetic
progressions with fixed gap size. It is a consequence of the next theorem
that neither the assumption of fixed gap size, nor that of unbounded lengths,
can be removed. For $\alpha>0$, the situation is less clear. While the fixed
gap condition still cannot be dropped, we do not know whether the condition
$l_k/a_k^\alpha\to\infty$ is necessary. A detailed discussion on this is given
in Section \ref{subsec:neccesity}. Then, in Section \ref{subsec:shashi} we comment on the potential role of an analogue of \eqref{eqnA01} with restricted coefficients.

Our next result describes classes of sets $\calN$ that admit uniqueness pairs
$(\Lambda, \calN)$, where the points in $\Lambda$ can go to the boundary
arbitrarily quickly. Moreover, the uniqueness pairs constructed are stable
under sufficiently small perturbations of the points in $\Lambda$. Before
stating the result, we introduce some terminology.

Let $\calN=\{n_k\}_{k\geq1}$ be a subset of $\Z_+$, written in increasing
order. For $0<\alpha\leq1$, we say that $\calN$ is \emph{$\alpha$-sparse} if
\begin{equation}\label{eq:alpha-sparse}
\sum_{k\geq1} \frac{1}{n_k^\alpha} < \infty.
\end{equation}
We say that $\calN$ is \emph{$0$-sparse} if it is finite or satisfies
\begin{equation}\label{eq:zero-sparse}
\sum_{k\geq0} \frac{n_k^2}{n_{k+1}^2} < \infty.
\end{equation}
Note that for $\alpha>0$, \eqref{eq:alpha-sparse} is equivalent to the condition
$$
\sum_{k\geq0} \frac{n_k^{2-\alpha}}{n_{k+1}^2} < \infty.
$$
From this perspective, \eqref{eq:zero-sparse} is the natural analogue at the
endpoint $\alpha=0$. 

For $0\leq\alpha\leq1$, we say that $\calN$ is \emph{locally $\alpha$-sparse}
if for each $n\in\calN$, there exists $q\in\N$ such that the set
$\{m\in\calN:\; m\equiv n \!\!\mod q\}$, written in increasing order, is $\alpha$-sparse. The term ``locally'' refers to the Furstenberg
topology on $\Z$ -- that is, the topology generated by the collection of
infinite arithmetic progressions. To be precise, $\calN$ is locally
$\alpha$-sparse if every element of $\calN$ has a Furstenberg neighbourhood
whose intersection with $\calN$ is $\alpha$-sparse.

We discuss several further examples in Section \ref{subsec:sparse-examples},
but for now we mention that the sequence of primes and the sequence
$\{2^k\}_{k\geq0}$ are locally $0$-sparse, and the sequence $\{k^d\}_{k\geq1}$
is locally $\alpha$-sparse precisely if $\alpha>1/d$.

\begin{theorem}\label{thm:Furstenberg}
Let $0\leq\alpha\leq1$ and let $\calN$ be a subset of $\Z_+$ which is
locally $\alpha$-sparse. Then for each increasing sequence $\{t_k\}_{k\geq1}
\subseteq(0,1)$, there exists $\Lambda=(\lambda_k)_{k\geq1}\subseteq\D$
such that $|\lambda_k|\geq t_k$ for every $k\geq1$ and $(\Lambda, \calN)$ is
a uniqueness pair for $\calD_\alpha$. Moreover, if $\wt{\Lambda}\subseteq\D$
satisfies
\begin{equation}\label{eq:perturbation}
\inf_{\mu\in\wt{\Lambda}}\rho(\mu, \lambda_k)
=
o\left(\sqrt{1-|\lambda_k|}\right),
\quad k\to\infty,
\end{equation}
where $\rho$ is the pseudohyperbolic metric, then $(\wt{\Lambda}, \calN)$ is
also a uniqueness pair.
\end{theorem}

\subsection{The case when $\calN$ and $\Lambda$ are finite}

The problem of characterizing the uniqueness pairs $(\Lambda, \calN)$ is
non-trivial even when $\calN$ and $\Lambda$ are finite. Suppose that
$\calN=\{n_1, \dots, n_l\}$ and $\Lambda=\{\lambda_1, \dots, \lambda_m\}$.
It is a classical fact that if $m>n_l$, then $(\Lambda, \calN)$ is a uniqueness
pair; this is a simple consequence of the Lagrange interpolation theorem.
In addition, if $m<l$, then $(\Lambda, \calN)$ is not a uniqueness pair.
Indeed, the dimension of $\vspan\{z^{n_k}\}_{k=1}^l$ is $l$, so if $m<l$, the
linear map $f\mapsto\{f(\lambda_j)\}_{j=1}^m$ from $\vspan\{z^{n_k}\}_{k=1}^l$
into $\C^m$ must have a non-trivial kernel. The following example illustrates that a simple description using only the sizes of
$\calN$ and $\Lambda$ is not available.

\begin{example}
Let $\calN=\{1,2,4\}$. Then for $\Lambda=\{1/2, -1/2, i/2\}$, we have that
$(\Lambda, \calN)$ is a uniqueness pair. Indeed for any function
$f(z)=az+bz^2+cz^4$, we have that $f(1/2)-f(-1/2)=a$ and $f(1/2)-f(i/2) =
a(1-i)/2 + b/2$, and so if $f$ vanishes on $\Lambda$, we must have $a=b=c=0$.
However, for $\Lambda=\{1/2, -1/2, 0\}$, $(\Lambda, \calN)$ is not a uniqueness
pair, as the function $f(z) = z^2-4z^4 = z^2(1-2z)(1+2z)$ readily demonstrates.
\end{example}

A general type of interpolation required for finite problems of this kind is known as Birkhoff interpolation, and it has a long history\cite{Birkhoff, Polya, Schoenberg}. Whenever $\Lambda$ and $\calN$ have the same cardinal and the nodes are real, the techniques in \cite{AtkinsonSharma} provide a somewhat satisfactory partial answer, not in the form of a closed formula but rather an algorithm. We suspect a complex analogue holds
 but we decided not to pursue this finite problems any further here.

\subsection{Notation and conventions}

Before continuing, we fix some basic notation that will be used in the
remainder of the paper.

\begin{enumerate}[label=\arabic*., itemsep=3pt]
\item
Throughout, $\Z_+=\{0,1,2,\dots\}$ and $\N=\{1,2,\dots\}$.
\item
If $A$ and $B$ are non-negative quantities, we  $A\lesssim B$ if
there exists a constant $C>0$ such that $A\leq CB$. We write $A\simeq B$ if
both $A\lesssim B$ and $B\lesssim A$ hold. The notation $A=O(B)$ and $A=o(B)$
has its usual meaning.
\item
For $f\in H^2$, $\wh{f}(k)$ denotes the $k^\text{\rm th}$ Taylor coefficient of $f$, and
$$
\supp\wh{f}=\{n\in\Z_+:\; \wh{f}(n)\neq0\}
$$
denotes its Taylor support.
\item
We denote by $S$ the shift operator on $H^2$,
$$
Sf(z)=zf(z),
$$
and by $S^*$ its adjoint, the backward shift.
\item
For $a,l,q\in\N$, we write $P(a,l,q)$ for the arithmetic progression of
length $l$, first term $a$, and gap size $q$, i.e.
$$
P(a,l,q)=\{a,\,a+q,\dots,a+(l-1)q\}.
$$

\item
The pseudohyperbolic metric on $\D$ is denoted by $\rho$ and is given by
$$
\rho(z,w)=\left|\frac{z-w}{1-\overline{w}z}\right|.
$$
\end{enumerate}


\section{Long arithmetic progressions}

\subsection{The case of infinite arithmetic progressions}

Before proving Theorem \ref{thm:long-AP}, we first address the simpler case
when $\calN$ contains an infinite arithmetic progression.

\begin{proposition}\label{prop:AP}
Fix $0\leq\alpha\leq1$, and let $\calN$ be a subset of $\Z_+$ which contains
an infinite arithmetic progression. Then $(\Lambda, \calN)$ is a uniqueness
pair for $\calD_\alpha$ if and only if $\Lambda$ is a uniqueness set for $\calD_\alpha$.
\end{proposition}

There may be some intersection between this proposition and a similar result for superharmonically weighted Dirichlet type spaces in \cite{Baoetal}, but a potential overlap is far from trivial.

\begin{proof}
Suppose $\calN$ contains the arithmetic progression $\{a+kq:\;k\in\Z_+\}$,
where $a,q\in\Z_+$ and $q\geq1$. Take $\Lambda$ to be a zero set for
$\calD_\alpha$. Then the set
$$
\wt{\Lambda} = \bigcup_{k<q} e^{i2\pi k/q}\Lambda
$$
is also a zero set for $\calD_\alpha$. Take a $f\in\calD_\alpha$,
$f\not\equiv0$, which vanishes at $\wt{\Lambda}$. Thus the function
$$
\wt{f}(z) = \frac{1}{q} \sum_{k<q} f(e^{i2\pi k/q} z)
$$
also vanishes at $\wt{\Lambda}$, and since it is invariant under rotation
by $2\pi/q$, its Taylor coefficients are supported on multiples of $q$.
Hence $S^a\wt{f} \in \calD_\alpha$ is a non-zero function whose Taylor coefficients are
supported on $\calN$ and which vanishes on $\Lambda$.
\end{proof}

\subsection{Proof of Theorem \ref{thm:long-AP}}
The proof proceeds by recursively constructing functions supported on $\calN$
that vanish on larger and larger finite subsets of $\Lambda$. The growth
condition on the arithmetic progressions, combined with the fact that the
points of $\Lambda$ approach the boundary sufficiently quickly, provides
sufficient norm control to obtain a non-trivial limit in $\calD_\alpha$.

\begin{proof}[Proof of Theorem \ref{thm:long-AP}]
It suffices to prove the claim for $q=1$. To see this, first note that we only
need to consider the case when all of the $P(a_k, l_k, q)$ are contained in
the same residue class mod $q$. With this in mind, we assume that each
$n\in P(a_k, l_k, q)$ satisfies $n\equiv m \!\mod q$ for some fixed $0\leq m<q$,
and set
$$
\wt a_k = (a_k-m)/q,
\quad
\calN_0 = \bigcup_{k\geq1} P(\wt a_k,l_k,1).
$$
Then $l_k/\wt{a}_k^\alpha \geq l_k/a_k^\alpha\to\infty$ as $k\to\infty$,
and for $f\in\calD_\alpha$ with $\supp\wh{f}\subseteq\calN$, there is some
$g\in\calD_\alpha$ with $\supp\wh{g}\subseteq\calN_0$ and such that
$f(z)=z^mg(z^q)$. Consequently, if $\{t_k\}_{k\ge1}\subset(0,1)$ has the
property that every sequence $\Lambda=\{\lambda_k\}_{k\ge1}\subset\D$
satisfying $|\lambda_k|\geq t_k$ for all sufficiently large $k$ fails to
be a uniqueness set for $\calN_0$, then the same conclusion holds for $\calN$
with $\{t_k^{1/q}\}_{k\ge1}$ in place of $\{t_k\}_{k\ge1}$.

From now on we suppose that $\calN$ is a union of intervals $[a_ k, b_k)$,
$k\geq0$, so that $l_k=b_k-a_k$. Observe that the condition \eqref{eq:progression}
implies that
$$
\sum_{a_k\leq n < b_k} \frac{1}{(n+1)^{\alpha}}
\geq
\frac{l_k}{(a_k+l_k)^{\alpha}}
\to \infty \quad\text{as}\quad k\to\infty.
$$
Indeed, if $l_k=O(a_k)$, then $l_k/(a_k+l_k)^{\alpha} \gs l_k/a_k^{\alpha}$,
otherwise $l_k/(a_k+l_k)^{\alpha} \sim l_k^{1-\alpha}$. Therefore, after
removing intervals if necessary, we can suppose that for each $k$,
\begin{equation}\label{eq:AP-interval-condition}
\sum_{a_k\leq n < b_k} \frac{1}{(n+1)^{\alpha}}
\geq k^2 b_{k-1}.
\end{equation}
Fix a sequence $\{\lambda_k\}_{k\geq1}$ in $\D$. We will construct an explicit
sequence of functions $f_k$, $k\geq0$, with the property that
$\supp\wh{f_k}\subseteq\calN$ and $f_k$ vanishes on $\lambda_1,\dots,\lambda_k$,
and show that if $|\lambda_k|\to1$ sufficiently fast, $f_k$ converges to
a non-zero function in $\calD_\alpha$.

Set $f_0(z) = z^{a_0}$, and for $k\geq1$, set $r_k=l_k/b_{k-1}$ and
$f_k(z)=f_{k-1}(1-p_k)$, where
$$
p_k(z) = \frac{1}{\lambda_k^{a_k}}
\left(\sum_{j<r_k}
\frac{|\lambda_k|^{2jb_{k-1}}}{(a_k+jb_{k-1})^\alpha}\right)^{-1}
z^{a_k}\sum_{j<r_k}
\frac{(\wb{\lambda_k}z)^{jb_{k-1}}}{(a_k+jb_{k-1})^\alpha}.
$$
Observe that for each $k$, $p_k(\lambda_k)=1$, so
$$
f_k=f_0\prod_{j=1}^{k}(1-p_j)
$$
vanishes on $\lambda_1,\dots,\lambda_k$. In addition, if $f_{k-1}$ is a
polynomial of degree at most $b_{k-1}-1$, the Taylor coefficients of
$f_{k-1}p_k$ are supported on $[a_{k}, b_{k})$, from which it follows by
induction that $\supp\wh{f_k} \subseteq\calN$.

Since $f_k$ and $f_{k-1}$ share the same coefficients up to the degree of $f_{k-1}$, the functions $f_k$ converge in $\calD_\alpha$ precisely if $\|f_k\|_\alpha$
is bounded uniformly in $k$. Before estimating $\|f_k\|_\alpha$, we note that
the Taylor coefficients of the functions $S^{a_k}S^{jb_{k-1}}f_{k-1}$,
$j=0,1,\dots b_{k-1}-1$, are supported on disjoint sets. Moreover,
$$
\|S^{a_k}S^{jb_{k-1}}f_{k-1}\|_{\alpha}^2
=
\sum_{n\leq b_{k-1}} |\wh{f}(n)|^2 (n+1+ a_k+jb_{k-1})^\alpha
\leq
(a_k+jb_{k-1})^\alpha \|f_{k-1}\|_{\alpha}^2.
$$
It follows that $\|f_{k-1}p_k\|_\alpha \leq \|f_{k-1}\|_\alpha\|p_k\|_\alpha$, and hence
\begin{equation}\label{eq:norm-f_k}
\|f_k\|_\alpha^2
\leq (1+\|p_k\|_\alpha^2)\|f_{k-1}\|_\alpha^2
\leq \|f_0\|_\alpha^2\prod_{1\leq j\leq k} (1+\|p_j\|_\alpha^2).
\end{equation}
It is straightforward to check that for each $k$,
\begin{equation}\label{eq:norm-p_k}
\|p_k\|_\alpha^2 = \frac{1}{|\lambda_k|^{2a_k}}
\left(\sum_{j<r_k}
\frac{|\lambda_k|^{2jb_{k-1}}}{(a_k+jb_{k-1})^\alpha}\right)^{-1}.
\end{equation}
Note that $\|p_k\|_\alpha^2$ is monotonically decreasing as $|\lambda_k|\to1$.
Choose $t_k<1$ such that
$$
t_k^{2a_k}\sum_{j<r_k}
\frac{t_k^{2jb_{k-1}}}{(a_k+jb_{k-1})^\alpha}
\geq \frac{1}{2}\sum_{j<r_k}
\frac{1}{(a_k+jb_{k-1})^\alpha}.
$$
Then taking $|\lambda_k|\geq t_k$ we get
\begin{align}\label{eq:norm-p_k2}
\|p_k\|_\alpha^2
&\ls
\left(\sum_{j<r_k}\frac{1}{(a_k+jb_{k-1})^\alpha}\right)^{-1} \\ \nonumber
&\leq
\left(\frac{1}{b_{k-1}}\sum_{n<b_{k-1}}\sum_{j<r_k}
\frac{1}{(n+ a_k+jb_{k-1})^\alpha}\right)^{-1}
\ls
b_{k-1}\left(\sum_{a_k\leq j<b_k}\frac{1}{j^\alpha}\right)^{-1}.
\end{align}

Therefore, if $|\lambda_k|\geq t_k$ for all sufficiently large $k$, it
follows from \eqref{eq:AP-interval-condition}, \eqref{eq:norm-f_k} and
\eqref{eq:norm-p_k2} that
$$
\lim_{k\to\infty}\|f_k\|_\alpha^2
\ls
\prod_{k\geq1} (1+1/k^2)
\ls
\exp\left(\sum_{k\geq1} 1/k^2\right) <\infty.
$$
\end{proof}

\subsection{Necessity of conditions in Theorem \ref{thm:long-AP}}
\label{subsec:neccesity}

We now examine the extent to which the assumptions in Theorem
\ref{thm:long-AP} are necessary. Our approach will be to construct unions
of arithmetic progressions which are locally $\alpha$-sparse, and then apply Theorem \ref{thm:Furstenberg} to conclude that this set does not satisfy the
conclusions of Theorem \ref{thm:long-AP} for $\calD_\alpha$.

We first show that the assumption of fixed gap size cannot be removed. Let
$p_k$ denote the $k^\text{th}$ prime and set $q_k=\prod_{j\leq k}p_j$.
Choose integers $a_k$ and $l_k$ recursively so that $a_{k+1}>a_k+l_k$, $l_k/a_k\to\infty$, and $a_k$ is divisible by $p_k$, but not by any of the
primes $p_1,\dots,p_{k-1}$. Observe that this can be done since at each stag
there are arbitrarily large integers satisfying the required divisibility
conditions. Set
$$
\calN=\bigcup_{k\geq1}P(a_k,l_k,q_k).
$$
Then $\calN$ contains arbitrarily long arithmetic progressions, and
the condition $l_k/a_k^\alpha\to\infty$ holds for every $0\leq\alpha\leq1$.
However, $\calN$ is locally $0$-sparse: if $n\in P(a_k,l_k,q_k)$, then
$$
\{m\in\calN:\; m\equiv n \!\!\mod p_k\}
$$
is finite. Thus the conclusion of Theorem \ref{thm:long-AP} fails for this
$\calN$, for every $0\leq\alpha\leq1$.

We now turn to the case $\alpha=0$, and show that the condition
$l_k\to\infty$ is also necessary. Fix $l\in\N$. We construct intervals $[a_k,a_k+l)$,
$k\geq1$, such that
$$
\calN =\bigcup_{k\geq1}[a_k,a_k+l)
$$
is locally $0$-sparse. The construction is recursive. At the $k$-th stage,
choose new primes $q_{k,0},q_{k,1},\dots,q_{k,l-1}$, all larger than $l$ and
larger than all previously chosen elements of $\calN$. We then choose $a_k$,
using the Chinese remainder theorem, so that for $0\leq j <l$ and every
previously chosen prime $q_{m,i}$, $a_k$ satisfies
\begin{align*}
a_k+j &\equiv 0\mod{q_{k,j}}, \\
a_k+j &\not\equiv 0\mod{q_{m,i}}.
\end{align*}
As before, we may choose $a_k$ arbitrarily large; in particular, we may
ensure that $a_k>a_{k-1}+l$. With this choice, each set
$\{n\in\calN:\; n\equiv a_k+j \mod q_{k,j}\}$ is finite. Hence $\calN$ is
locally $0$-sparse. Therefore the hypothesis $l_k\to\infty$ in Theorem
\ref{thm:long-AP} is necessary when $\alpha=0$.

For $\alpha>0$, the situation is less clear. Let
$$
\calN=\bigcup_{k\geq1}[a_k,a_k+l_k),
$$
and suppose that $l_k=O(a_k^\beta)$ for some $\beta<\alpha$. Then
$$
\sum_{n\in\calN} n^{-\alpha}
\lesssim
\sum_{k\geq1} l_k a_k^{-\alpha}
\lesssim
\sum_{k\geq1} a_k^{\beta-\alpha}
<\infty,
$$
and hence $\calN$ is $\alpha$-sparse. Consequently, the conclusion of Theorem
\ref{thm:long-AP} fails for $\calD_\alpha$ in this case. This leads naturally
to the following question.

\begin{problem}
Let $0<\alpha\leq1$. Does there exist a set $\calN\subseteq\Z_+$ containing
arithmetic progressions
$$
P(a_k,l_k,q),
\qquad k\geq1,
$$
for some fixed $q\in\N$, with $l_k\simeq a_k^\alpha$, such that the conclusion
of Theorem \ref{thm:long-AP} fails for $\calD_\alpha$?
\end{problem}

Note that if $l_k\simeq a_k^\alpha$, then $\calN$ is not locally
$\alpha$-sparse. Indeed, for $m,q\in\N$ and sufficiently large $k$, the
interval $[a_k,a_k+l_k)$ contains at least
$l_k/(2q)$ integers congruent to $m\!\mod q$. Since every
$n\in[a_k,a_k+l_k)$ satisfies $n\simeq a_k$, we obtain
$$
\sum_{\substack{a_k \leq n < a_k+l_k\\ n\equiv m \!\!\mod q}} n^{-\alpha}
\gtrsim
\frac{l_k}{q}a_k^{-\alpha}
\gtrsim
\frac1q.
$$
Summing over $k$, it follows that
$$
\sum_{\substack{n\in\calN\\ n\equiv m \!\!\mod q}} n^{-\alpha}
=\infty.
$$
Thus $\calN$ is not locally $\alpha$-sparse. Consequently, one cannot use
Theorem \ref{thm:Furstenberg} to show that such sets fail to satisfy the
conclusion of Theorem \ref{thm:long-AP}.

\subsection{Shapiro-Shields and Carleson type conditions on restricted kernels}\label{subsec:shashi}

Our concept of uniqueness with restrictions on coefficients can be approached as a usual type of uniqueness within the following space. Let $\alpha \in \R$ and $\calN \subseteq \N$. Then $\calN$ induces a subspace of $\calD_\alpha$ that we denote
$$
\calD_{\alpha, \calN}=\wb{\vspan}\{z^k: k \in \calN \}.
$$
Since $\calD_{\alpha, \calN}$ is a closed subspace of $\calD_\alpha$, it is a reproducing kernel Hilbert space in itself, whose kernels $k^{\alpha}_{\calN, \lambda}$ are the corresponding projections of the kernels in $\calD_\alpha$ in \eqref{eqn47}, that is \[k^{\alpha}_{\calN, \lambda}(z) = \sum_{k \in \calN} \frac{\overline{\lambda}^k z^k}{(k+1)^\alpha}.\]
 This leads to a possible construction in the spirit of Shapiro-Shields \cite{ShaShi} and we suspect that the role of the equation \eqref{eqnA01} may be replaced by the natural
$$
\sum_{\lambda\in\Lambda} \frac{1}{\|k^{\alpha}_{\calN,\lambda}\|_\alpha^2} = \infty.
$$
There is a sense in which the Shapiro-Shields condition is sharp, as shown by Carleson \cite{Carleson}; see also \cite{NRS} and \cite[Section 4.2]{EFKMR}. This is in some way reflected in our two type of Theorems, but we preferred to pursue arithmetic conditions rather than determinant estimates that depend on kernels for which the theory is still rather poor. We do not know whether $k^{\alpha}_{\calN,\lambda}$ will generate issues such as unprescribed zeros of the Shapiro-Shields functions, or linear dependences between the kernels.
However, we encourage the community to join the effort in this regard. Perhaps an interplay between techniques for studying boundary values in more abstract RKHS, like the second paper in Dahlin's thesis \cite{Dahlin}, and a strong control on zeros of kernels, in the spirit of \cite{Perala} yields a complete answer here.


\section{Locally sparse sequences}

\subsection{Examples of locally sparse sequences}\label{subsec:sparse-examples}

Here we present several classes of sequences that have been extensively
studied in the number theory literature.

\begin{enumerate}[wide, labelwidth=!, labelindent=0pt, itemsep=10pt, topsep=10pt]
\item \textit{Primitive sets.}
A set $\calN\subseteq\Z_+$ is called \emph{primitive} if, whenever
$m,n\in\calN$ and $m\mid n$, we must have $m=n$. Classical examples include
the primes, and more generally the sets of $k$-almost primes -- that is,
the integers with exactly $k$ prime factors (counted with multiplicity).
Every infinite primitive set is locally $0$-sparse. Indeed, if $n\in\calN$,
then the congruence class $0 \!\mod n$ intersects $\calN$ only at $n$ itself.
On the
other hand, many natural primitive sets are not $\alpha$-sparse for any
$0\leq\alpha\leq1$; for example, this is the case for the primes since
$$
\sum_{p \text{ prime}} \frac{1}{p} = \infty.
$$

\item \textit{Hardy--Littlewood--P\'olya sequences.}
Fix distinct primes $p_1,\dots,p_r$. The corresponding
Hardy--Littlewood--P\'olya sequence is the set
$$
\calN=\{p_1^{k_1}p_2^{k_2}\cdots p_r^{k_r}:\; k_1,\dots,k_r\geq0\},
$$
written in increasing order. These sequences are $\alpha$-sparse for every
$\alpha>0$, since
$$
\sum_{n\in\calN} n^{-\alpha}
=
\prod_{j=1}^r \sum_{k\geq0} p_j^{-\alpha k}
<\infty.
$$
They are also locally $0$-sparse. Indeed, fix
$n=p_1^{a_1}\cdots p_r^{a_r}\in\calN$, and set $q=p_1^{a_1+1}\cdots p_r^{a_r+1}$.
Suppose that $m=p_1^{b_1}\cdots p_r^{b_r}\in\calN$ satisfies
$m\equiv n\!\mod q$. Then, for each $j$, we have that
$m\equiv n\!\mod {p_j^{a_j+1}}$.
Let $s_j=\min(a_j,b_j)$, and write
$$
m'=m/p_j^{s_j},
\quad
n'=n/p_j^{s_j}.
$$
Since $s_j\leq a_j$, the above congruence implies that
$m'\equiv n'\!\mod{p_j}$. If $a_j\neq b_j$, then exactly one of $m'$ and $n'$
is divisible by $p_j$, while the other is not, which is impossible. Hence
$a_j=b_j$ for every $j$, and therefore $m=n$. Thus the congruence class
$n\!\mod q$ intersects $\calN$ only at $n$ itself, showing that
$\calN$ is locally $0$-sparse.

\item \textit{Divisibility chains.}
A sequence $\calN=\{n_k\}_{k\geq1}\subseteq\Z_+$ is called a divisibility
chain if
$$
n_k\mid n_{k+1}
\qquad\text{for all } k\geq1.
$$
Classical examples include the geometric progressions $\{a^k\}_{k\geq0}$,
the factorials $\{k!\}_{k\geq1}$, and the least common multiples
$\{\operatorname{lcm}(1,\dots,k)\}_{k\geq1}$. Every increasing divisibility
chain is locally $0$-sparse. Indeed, if $n_j\in\calN$, then every later term
in the sequence is divisible by $n_{j+1}$, and hence the congruence class
$n_j\!\mod{n_{j+1}}$ intersects $\calN$ only in the finite set
$\{n_1,\dots,n_j\}$. Moreover, every increasing divisibility chain is
$\alpha$-sparse for every $\alpha>0$, since $n_{k+1}\geq 2n_k$ for all $k$.
On the other hand, divisibility chains need not be $0$-sparse; for example,
the sequences $\{a^k\}_{k\geq0}$ and
$\{\operatorname{lcm}(1,\dots,k)\}_{k\geq1}$ are not $0$-sparse, whereas
$\{k!\}_{k\geq1}$ is.

\item \textit{Monomial orbits and Piatetski--Shapiro sequences.}
Fix $t>1$, and consider the sequence
$$
\calN = \{\lfloor k^t\rfloor\}_{k\geq1}.
$$
When $t\in\Z$, this is the monomial orbit $\{k^t\}_{k\geq1}$, while for
non-integer $t$ this is known as a Piatetski--Shapiro sequence. In either case,
$\calN$ is locally $\alpha$-sparse if and only if it is $\alpha$-sparse.
For integer $t$, this follows from the observation that, for any
$k,q\in\N$, the set
$$
\{n\in\calN:\; n\equiv k^t \!\mod q\}
$$
contains the sequence $\{(k+jq)^t\}_{j\geq1}$. For non-integer $t$, the
same conclusion follows from the equidistribution of Piatetski--Shapiro
sequences in residue classes \cite{DDMS}. Indeed, if $a\in\{0,\dots,q-1\}$ and
$$
\{n\in\calN:\; n\equiv a \!\mod q\}
=
\{n_{k_j}\}_{j\geq1},
$$
then
$$
|\{j:\; k_j\leq N\}| \sim N/q,
\quad N\to\infty.
$$
Hence $k_j\sim qj$ as $j\to\infty$, and therefore
$n_{k_j}=\lfloor k_j^t\rfloor\sim (qj)^t$.
Consequently, $\calN$ is locally $\alpha$-sparse precisely when $\alpha>1/t$.
\end{enumerate}

\begin{remark}
Given any $0\leq\alpha_0\leq1$, the preceding examples may be combined
to produce sets $\calN\subseteq\Z_+$ which are not $\alpha$-sparse for
any $0\leq\alpha\leq1$, but which are locally $\alpha$-sparse precisely
for $\alpha>\alpha_0$. For example, let
$$
\calN
=
\{p\geq3:\; p \text{ prime}\}
\cup
\{2\lfloor k^t\rfloor:\; k\geq1\},
$$
where $t>1$. Since $\calN$ contains the odd primes, it is not
$\alpha$-sparse for any $0\leq\alpha\leq1$. Moreover, since it contains
$\{2\lfloor k^t\rfloor\}_{k\geq1}$, it is not locally $\alpha$-sparse
for any $0\leq\alpha\leq1/t$.

On the other hand, if $p$ is an odd prime, then
$$
\{n\in\calN:\; n\equiv p \!\!\mod 2p\}
$$
is contained in
$$
\{p\}\cup\{2\lfloor k^t\rfloor:\; k\geq1\},
$$
while if $n=2\lfloor k^t\rfloor$, then
$$
\{m\in\calN:\; m\equiv n \!\!\mod 2\}
=
\{2\lfloor j^t\rfloor:\; j\geq1\}.
$$
In either case, the resulting set is $\alpha$-sparse when $\alpha>1/t$.
Hence $\calN$ is locally $\alpha$-sparse precisely when $\alpha>1/t$.
\end{remark}

The case of monomial orbits, and more generally orbits of integer
polynomials, is particularly intriguing. For the orbit $\{k^d\}_{k\geq1}$,
it is natural to ask whether the threshold $\alpha>1/d$ is intrinsic to the
problem, or merely a limitation of our methods. This leads to the following problem.

\begin{problem}
Let $d\geq2$, and let $\calN=\{k^d\}_{k\geq1}$ (or more generally, let
$\calN$ be the orbit of an integer polynomial of degree $d$). For which
$0\leq\alpha\leq 1/d$, if any, does the conclusion of Theorem \ref{thm:Furstenberg} hold for $\calD_\alpha$?
\end{problem}


\subsection{Proof of Theorem \ref{thm:Furstenberg}}
We begin with two preliminary Lemmas that address the case of $0$-sparse
sequences.

\begin{lemma}\label{lem:lacunary}
For each $k_0\geq1$, and each interval $I\subseteq [0,1]$, of length
$|I|\geq 2/n_{k_0+1}$, there exists $\theta\in I$ such that
$$
\dist(n_k\theta,\Z) = \inf_{m\in\Z}|n_k\theta-m| \leq n_k/n_{k+1}
\quad\text{for all} \;\; k>k_0.
$$
\end{lemma}

\begin{proof}
Set $\beta_k=n_k/n_{k+1}$.
Fix $k_0\geq1$ and an interval $I=I_{k_0}\subseteq[0,1]$ with
$|I|\geq 2\beta_{k_0}/n_{k_0}$. Since $1/n_{k_0+1}=\beta_{k_0}/n_{k_0}\leq|I_{k_0}|/2$, we can
find $j_{k_0+1}\in\Z$ such that
$$
[j_{k_0+1}/n_{k_0+1}, (j_{k_0+1}+1)/n_{k_0+1}]
\subseteq I_{k_0}.
$$
Then letting
$$
I_{k_0+1}=[(j_{k_0+1}-\beta_{k_0+1})/n_{k_0+1}, (j_{k_0+1}+\beta_{k_0+1})/n_{k_0+1}],
$$
we have that $I_{k_0+1}\subseteq I_{k_0}$.
Iterating this argument, we obtain a sequence $\{j_k\}_{k>k_0}$, so that
the closed intervals
$$
I_k = [(j_k-\beta_k)/n_k, (j_k+\beta_k)/n_k]
$$
are nested. The intersection of the $I_k$ consists of a single point and so
we take
$$
\theta\in\bigcap_{k\geq k_0}I_k.
$$
\end{proof}

\begin{lemma}\label{lem:0-sparse}
Assume $\calN=\{n_k\}_{k\geq1}$ is $0$-sparse. Then there exists a
sequence of finite sets $\Theta_k=\Theta_k(\calN)\subseteq\T$, $k\geq1$,
such that for every sequence $\{r_k\}_{k\geq1}\subseteq(0,1)$ with $r_k\to1$,
and every $f\in H^2$ with $\supp\wh{f}\subseteq\calN$, the following holds:
\begin{equation}\label{eq:sparse-norm-estimate}
\|f\|_{H^2}^2
\ls
\liminf_{k\to\infty}
\frac{1}{|\Theta_k|}\sum_{\zeta\in\Theta_k} |f(r_k \zeta)|^2.
\end{equation}
Moreover, if $\calN$ is contained in an arithmetic progression with gap size
$q$, then we can take each $\Theta_k(\calN)$ to be invariant under rotation
by $2\pi/q$.
\end{lemma}

\begin{proof}
For each $k\geq 1$ and $j=1,\dots, n_k$, we apply Lemma \ref{lem:lacunary}
with
$$
I=[j/n_k-1/n_{k+1},\; j/n_k+1/n_{k+1}]
$$
to obtain points $\theta_{k,j}$ such that $|\theta_{k,j}-j/n_k|\leq 1/n_{k+1}$
and $\dist(n_l\theta_{k,j},\Z) \leq n_l/n_{l+1}$ for $l>k$. Write
$\zeta_{k,j} = e^{2\pi i \theta_{k,j}}$ and set
$$
\Theta_k(\calN)
=
\{\zeta_{k,j}:\; j=1,\dots,n_k\}.
$$

Fix a sequence $\{r_k\}_{k\geq1}$ increasing to $1$, and take $f\in H^2$
whose Taylor coefficients are supported on $\calN$. For $k\geq1$, set
$$
f_k(z)=\sum_{j<k} \wh{f}(n_j) z^{n_j},
\quad g_k(z)=f(z)-f_k(z).
$$
Note that $f_k$ is a polynomial of degree at most $n_k-1$ with mean zero on
$\T$, and so
$$
\sum_{j < n_k} f_k(r_k e^{i2\pi j/n_k}) = 0.
$$
In addition, for $l\leq k$ and $j\leq n_k$,
$$
|\zeta_{k,j}^{n_l}-e^{i2\pi jn_l/n_k}|
\ls n_l|\theta_{k,j} - j/n_k|
\leq n_k/n_{k+1}.
$$
Using these, and applying Cauchy-Schwarz, we see that
\begin{align}\label{eq:truncated}
\nonumber
\Abs{\frac{1}{n_k}\sum_{j\leq n_k} f_k(r_k \zeta_{k,j})}
&=
\Abs{\frac{1}{n_k}\sum_{j\leq n_k}
\sum_{l < k} \wh{f}(n_l) r_k^{n_l}(\zeta_{k,j}^{n_l}-e^{i2\pi jn_l/n_k})} \\ \nonumber
&\leq
\frac{n_k}{n_{k+1}}\sum_{l < k} |\wh{f}(n_l)| r_k^{n_l} \\ 
&\ls
\|f\|_{H^2}
\frac{n_k}{n_{k+1}}\sqrt{k}.
\end{align}
For $l>k$, we have $|\zeta_{k,j}^{n_l}-1|\leq n_l/n_{l+1}$, and so for
each $j=1,\dots, n_k$ we have
\begin{align}\label{eq:tail}
\nonumber
\Abs{\left(g_k(r_k \zeta_{k,j}) - g_k(r_k)\right)}
&=
\Abs{\sum_{l>k} \wh{f}(n_l) r_k^{n_l}(\zeta_{k,j}^{n_l}-1)} \\
&\leq
\sum_{l>k} |\wh{f}(n_l)| r_k^{n_l}\frac{n_l}{n_{l+1}}
\ls
\|f\|_{H^2}\left(\sum_{l>k} \frac{n_l^2}{n_{l+1}^2}\right)^{1/2}.
\end{align}
The fact that $\calN$ is $0$-sparse implies that, by taking $k$ along
a suitable subsequence, each of the expressions in \eqref{eq:truncated}
and \eqref{eq:tail} converge to zero. It follows that
$$
|g_k(r_k)| 
= 
\Abs{\frac{1}{n_k}\sum_{j\leq n_k} g_k(r_k \zeta_{k,j})} + o(1)
=
\Abs{\frac{1}{n_k}\sum_{j\leq n_k} f(r_k \zeta_{k,j})} + o(1).
$$ 
So for each $j$,
$$
|g_k(r_k \zeta_{k,j})|^2
\ls
\Abs{\frac{1}{n_k}\sum_{j\leq n_k} f(r_k \zeta_{k,j})}^2 + o(1)
\ls
\frac{1}{n_k}\sum_{j\leq n_k} |f(r_k \zeta_{k,j})|^2 + o(1),
$$
and
$$
|f_k(r_k \zeta_{k,j})|^2
=
|f(r_k \zeta_{k,j})-g_k(r_k \zeta_{k,j})|^2
\ls
|f(r_k \zeta_{k,j})|^2
+
\frac{1}{n_k}\sum_{j\leq n_k} |f(r_k \zeta_{k,j})|^2 + o(1).
$$
Then by the Kadets $1/4$ theorem for polynomials \cite{Marzo-Seip}, we have
$$
\sum_{j<k} |\wh{f}(n_j)|^2 r_k^{2n_j}
\ls
\frac{1}{n_k}\sum_{j\leq n_k} |f_k(r_k \zeta_{k,j})|^2
\ls
\frac{1}{n_k}\sum_{j\leq n_k} |f(r_k \zeta_{k,j})|^2 + o(1),
$$
from which we conclude the estimate \eqref{eq:sparse-norm-estimate}.

Finally, observe that if each $n\in\calN$ satisfies $n\equiv m \!\! \mod q$
for some $q\in\N$ and $0\leq m<q$, then the set
$\wt{\calN}=\{(n-m)/q:\;n\in\calN\}$ is also $0$-sparse, so we can use the
construction above to get sets $\Theta_k(\wt{\calN})$. Moreover, if
$\supp\wh{f}\subseteq\calN$, there is some $g\in H^2$ with
$\supp\wh{g}\subseteq\wt{\calN}$ and such that $f(z)=z^mg(z^q)$. Then if
we redefine $\Theta_k(\calN)$ by
$$
\Theta_k(\calN) =\{\zeta\in\T:\; \zeta^q\in\Theta_k(\wt{\calN})\},
$$
we have that $\Theta_k(\calN)$ is invariant under rotation by $2\pi/q$ and
\begin{align*}
\|f\|_{H^2}^2
=
\|g\|_{H^2}^2
&\ls
\liminf_{k\to\infty}
\frac{1}{|\Theta_k(\wt{\calN})|}\sum_{\zeta\in\Theta_k(\wt{\calN})}
|g(r_k \zeta)|^2 \\
&=
\liminf_{k\to\infty}
\frac{1}{|\Theta_k(\calN)|}\sum_{\zeta\in\Theta_k(\calN)}
|f(r_k \zeta)|^2.
\end{align*}
\end{proof}

We will also need two simple lemmas to allow us to pass from
sparse to locally sparse sequences.


\begin{lemma}\label{lem:decomposition}
Let $0\leq\alpha\leq1$, and let $\calN\subseteq\Z_+$ be locally
$\alpha$-sparse. Then $\calN$ admits a decomposition
\begin{equation}\label{eq:sparse-slices2}
\calN=\bigcup_{k\geq1}\calN_k,
\end{equation}
where the sets $\calN_k$ are pairwise disjoint and $\alpha$-sparse.
Moreover, there exist integers $Q_k\in\N$ and $0\leq l_k<Q_k$,
$k\geq1$, such that $Q_k\mid Q_{k+1}$ and
\begin{equation}\label{eq:sparse-slices}
\calN_k
=
\{n\in\calN:\; n\equiv l_k \!\!\mod Q_k\}
\end{equation}
whenever $\calN_k$ is non-empty.
\end{lemma}

\begin{proof}
If $\calN$ is $\alpha$-sparse, we can simply take
$$
\calN_1=\calN,\qquad \calN_k=\emptyset, \quad k>1,
$$
and set $Q_k=l_k=1$ for all $k\geq1$. So we will assume that $\calN$ is not
$\alpha$-sparse. In this case, we construct the sequences $\{l_k\}_{k\geq1}$
and $\{Q_k\}_{k\geq1}$ inductively, and define $\calN_k$ by \eqref{eq:sparse-slices}.

To begin, choose $n_1$ to be the smallest element of $\calN$. Since $\calN$
is locally $\alpha$-sparse, there exists $Q_1\in\N$ such that 
$$
\{n\in\calN:\; n\equiv n_1 \!\!\mod Q_1\}
$$
is $\alpha$-sparse. Let $l_1$ to be the residue class of $n_1$ modulo $Q_1$.

Suppose now that $l_1,\dots,l_{k-1}$ and $Q_1,\dots,Q_{k-1}$ have been chosen. Since $\calN$ is not $\alpha$-sparse, there exists $n\in\calN$ such that $n\not\equiv l_j \!\!\mod Q_j$, for all $j<k$. Let $n_k$ be the smallest such element. Choose $q_k\in\N$ such that
$$
\{n\in\calN:\; n\equiv n_k \!\!\mod q_k\}
$$
is $\alpha$-sparse. Then set $Q_k=\lcm(Q_{k-1},q_k)$ and $l_k$ to be the
residue class of $n_k$ modulo $Q_k$.

Note that, with $\calN_k$ given by \eqref{eq:sparse-slices}, we have
$$
\calN_k\subseteq\{n\in\calN:\; n\equiv n_k \!\!\mod q_k\},
$$
and so each $\calN_k$ is $\alpha$-sparse. Moreover, it's clear that
\eqref{eq:sparse-slices2} holds. Finally, to see that the sets $\calN_k$,
are pairwise disjoint, note that if $j<k$, then $Q_j\mid Q_k$. So if
$n\equiv l_k \!\!\mod Q_k$, we must have that 
$$
n\equiv l_k \not\equiv l_j \!\!\mod Q_j.
$$
This completes the proof.
\end{proof}

\begin{lemma}\label{lem:modular}
Fix $Q\in\N$ and $\zeta_0\in\T$. Let $\zeta_j=e^{i2\pi j/Q}\zeta_0$, $j=0,\dots,Q-1$. For $f\in H^2$ and $l=0,\dots,Q-1$ set
$$
F_l(z) = \sum_{k\equiv l \mathrm{\;mod\;} Q} \wh{f}(k) z^k.
$$
Then for any $r<1$ and $0\leq l, m <Q$,
$$
F_l(r\zeta_m) = \frac{1}{Q} \sum_{j<Q} (\zeta_m\wb{\zeta_j})^l f(r\zeta_j).
$$
\end{lemma}

\begin{proof}
For each $r<1$ and $0\leq l <Q$ we have
$$
\frac{1}{Q} \sum_{j<Q} F_l(r\zeta_j)
=
\sum_{k\equiv l \mathrm{\;mod\;} Q} \wh{f}(k)r^k
\frac{1}{Q} \sum_{j<Q} \zeta_j^k
=
\begin{cases}
F_l(r\zeta_0), \quad l=0, \\
0, \quad \text{otherwise}.
\end{cases}
$$
It follows that
$$
\frac{1}{Q} \sum_{j<Q} f(r\zeta_j)
=
\sum_{l<Q}\frac{1}{Q} \sum_{j<Q} F_l(r\zeta_j)
=F_0(r\zeta_0).
$$
Since $F_0$ identifies the point $r\zeta_0, \dots, r\zeta_{Q-1}$, this proves
the claim with $l=0$.

For $l\neq0$, we observe that
$$
F_l(z) = z^l\sum_{k\equiv 0 \mathrm{\;mod\;} Q} \wh{S^{*l}f}(k) z^k,
$$
and
$$
\frac{1}{Q} \sum_{j<Q} S^{*l}f(r\zeta_j)
=
\frac{1}{Q} \sum_{j<Q} r^{-l}\wb{\zeta_j}^{l}
\left(f(r\zeta_j) - \sum_{k<l}\wh{f}(k)r^k\zeta_j^k\right) 
=
\frac{1}{Q} \sum_{j<Q} r^{-l}\wb{\zeta_j}^{l}f(r\zeta_j).
$$
The conclusion now follows from applying the $l=0$ case to the function
$S^{*l}f$.
\end{proof}


In the proof of Theorem \ref{thm:Furstenberg}, we use the following identity for the
reproducing kernels in $H^2$:
\begin{equation}\label{eq:metric}
\|k^0_{\lambda} - k^0_{\mu}\|_{H^2}^2
=
\rho(\lambda,\mu)^2\frac{1-|\lambda|^2|\mu|^2}{(1-|\lambda|^2)(1-|\mu|^2)}.
\end{equation}
This can be verified by a direct computation.

\begin{proof}[Proof of Theorem \ref{thm:Furstenberg}]
Fix $\calN\subseteq\Z_+$ which is locally $\alpha$-sparse. We will construct
a sequence of finite sets $\Gamma_k\subseteq\T$, $k\geq1$, with the property
that $(\Lambda, \calN)$ is a uniqueness pair for $\calD_\alpha$, whenever
$$
\Lambda=\bigcup_{k\geq1} r_k\Gamma_k,
\quad
\{r_k\}_{k\geq1}\subseteq(0,1).
$$
Then writing $\Lambda=\{\lambda_k\}$ with $|\lambda_1|\geq|\lambda_2|\geq\dots$,
it is clear that given any sequence $\{t_k\}_{k\geq1}\subseteq(0,1)$,
we can take $r_k\to1$ sufficiently fast to ensure that $|\lambda_k|\geq t_k$
for every $k$. To this end, in the remainder of the proof, we fix an arbitrary
sequence $\{r_k\}_{k\geq1}\subseteq(0,1)$ tending to $1$, and
$f\in\calD_\alpha$ whose Taylor coefficients are supported on $\calN$.

For $k\geq1$, let $Q_k$, $l_k$ and $\calN_k$ be given by Lemma
\ref{lem:decomposition}, so that the sets $\calN_k$ are pairwise disjoint and $\alpha$-sparse, and \eqref{eq:sparse-slices2} and
\eqref{eq:sparse-slices} hold. We then decompose $f$ as
$$
f(z) =\sum_{k\geq1} f_k(z),
\qquad
f_k(z)
=
\sum_{n\in\calN_k} \wh{f}(n) z^n.
$$

First suppose that $\alpha>0$. If $\calN$ is $\alpha$-sparse, we can take
$\Gamma_k$ to be any collection of finite sets satisfying
\begin{equation}\label{eq:mesh}
\sup_{\zeta\in\T} \dist(\zeta, \Gamma_k) \to 0, \quad k\to\infty.
\end{equation}
Otherwise, we can take $\Gamma_k$ to be any finite subset of $\T$ which is
invariant under rotation by $2\pi/Q_k$. Note that in the case when $\calN$
is not $\alpha$-sparse, we necessarily have $Q_k\to\infty$, and so
\eqref{eq:mesh} automatically holds. Since each $\calN_j$ is $\alpha$-sparse,
each function $f_j$ extends continuously to $\wb{\D}$. Recall that for
$k\geq j$, $Q_j$ divides $Q_k$, and so $\Gamma_k$ is invariant under
rotation by $2\pi/Q_j$. It then follows from Lemma \ref{lem:modular} that
\begin{align}\label{eq:alpha-positive}
\|f_j\|_\infty 
&= \lim_{k\to\infty} \max_{\zeta\in\Gamma_k}|f_j(r_k\zeta)| \\ \nonumber
& \qquad\leq \lim_{k\to\infty} \max_{\zeta\in\Gamma_k}
\frac{1}{Q_j}\sum_{l<Q_j}|f(r_ke^{i2\pi l/Q_j}\zeta)|
\leq \lim_{k\to\infty} \max_{\zeta\in\Gamma_k}|f(r_k\zeta)|.
\end{align}

Next, suppose that $\alpha=0$. In this case we take
$$
\Gamma_k=\bigcup_{j\leq k} \Theta_k(\calN_j),
$$
where $\Theta_k(\calN_j)$ is given by Lemma \ref{lem:0-sparse} and chosen
to be invariant under rotation by $2\pi/Q_j$. We use the convention that
$\Theta_k(\emptyset)=\emptyset$. Then it is an immediate consequence
of Lemmas \ref{lem:0-sparse} and \ref{lem:modular} that for each $j$ with
$\calN_j$ non-empty, we have
\begin{equation}\label{eq:alpha-zero}
\|f_j\|_{H^2}^2
\ls
\liminf_{k\to\infty}
\frac{1}{|\Theta_k(\calN_j)|}\sum_{\zeta\in\Theta_k(\calN_j)} 
|f_j(r_k \zeta)|^2
\ls
\liminf_{k\to\infty}
\frac{1}{|\Theta_k(\calN_j)|}\sum_{\zeta\in\Theta_k(\calN_j)} 
|f(r_k \zeta)|^2.
\end{equation}
We conclude that in either case, if $f$ vanishes on $r_k\Gamma_k$ for all
$k$, then $f=0$ -- that is $(\Lambda, \calN)$ is a uniqueness pair, with
$\Lambda=\cup_k r_k\Gamma_k$.

It only remains to show that if $\wt{\Lambda}$ is a suitably small
perturbation of $\Lambda$, then $(\wt{\Lambda}, \calN)$ is also a uniqueness
pair. If $\wt{\Lambda}$ has an accumulation point in $\D$, then the conclusion
is immediate, so we may assume that it has none. Write
$\Lambda=\{\lambda_k\}_{k\geq1}$ and suppose that $\wt\Lambda$ satisfies
\eqref{eq:perturbation}. For $\lambda_k\in\Lambda$, let $\mu_k$ be a point
in $\wt{\Lambda}$ satisfying
$$
\rho(\mu_k, \lambda_k)=\inf_{\mu\in\wt{\Lambda}}\rho(\mu,\lambda_k).
$$
Then if $f$ vanishes on $\wt{\Lambda}$, we have
\begin{align*}
|f(\lambda_k)| &= |f(\lambda_k) - f(\mu_k)| \\
&\leq  \|k_{\lambda_k} - k_{\mu_k}\| \|f\|_{H^2} \\
&= \rho(\mu_k, \lambda_k)
\sqrt{\frac{1-|\lambda_k|^2|\mu_k|^2}{(1-|\lambda_k|^2)(1-|\mu_k|^2)}}\|f\|_{H^2} \\
&\ls \frac{\rho(\lambda_k, \mu_k)}{\sqrt{1-|\lambda_k|}}\|f\|_{H^2}.
\end{align*}
Hence $f(\lambda_k)\to0$, and so \eqref{eq:alpha-positive} or
\eqref{eq:alpha-zero} (depending on whether $\alpha$ is positive or zero,
respectively) implies that $f=0$, which completes the proof.
\end{proof}


\end{document}